\documentclass{article}
\usepackage{amssymb,amsmath,amsthm}

\newtheorem{proposition}{Proposition}[section]
\newtheorem{theorem}{Theorem}[section]
\newtheorem{definition}{Definition}[section]
\newtheorem{remark}{Remark}[section]
\newtheorem{corollary}{Corollary}[section]
\newtheorem{example}{Example}[section]

\begin{document}


\begin{center}
\textbf{\Large A simple characterization of homogeneous Young measures and weak $L^1$ convergence of their densities}
\end{center}
\begin{center}
\textbf{{\large Piotr Pucha{\l}a}}\\
Institute of Mathematics, Czestochowa University of Technology, Armii Krajowej 21, 42-200 Cz\c{e}stochowa, Poland\\
Email:    piotr.puchala@im.pcz.pl, p.st.puchala@gmail.com
\end{center}


\begin{abstract}
We formulate a simple characterization
of homogeneous Young measures associated with measurable functions. It is based on the notion of the quasi-Young measure introduced in the previous article published in this Journal. First, homogeneous Young measures associated with the measurable functions are recognized as the constant mappings defined  on the domain of the underlying function with values in the space of probability measures on the range of these functions. Then the characterization of homogeneous Young measures via image measures is formulated. Finally, we investigate the connections between weak convergence of the homogeneous Young measures understood as elements of the Banach space of scalar valued measures and the weak$^\ast$ $L^1$ sequential convergence of their densities. A scalar case of the smooth functions and their Young measures being Lebesgue-Stieltjes measures is also analyzed.\\

\textbf{keywords}: homogeneous Young measures; weak convergence of measures; weak convergence of functions; optimization

 \bigskip
\end{abstract}

\section[]{Introduction}
\quad Young measures is an abstract measure theoretic tool which has arisen in the context of minimizing bounded from below integral functionals not attaining their infima. The often used direct method relays on constructing minimizing consi\-de\-red functional $\mathcal J$ sequence $(u_n)$ of functions. This sequence is converging to some function $u_0$ in an appropriate (usually weak$^\ast$) topology, while the sequence $\mathcal J(u_n)$ of real numbers converges to $\inf\mathcal J$. However, in general it is not true that the sequence $(f(x,u_n(x),\nabla u_n(x)))$ of compositions of the integrand $f$ of $\mathcal J$ with elements of $(u_n)$ converges to $f(x,u_0(x),\nabla u_0(x)))$. This is connected with the lack of quasiconvexity of the integrand $f$ with respect to its third variable. This situation is often taking place in elasticity theory (here $f$ is the density of the internal energy of the displaced body). See for example \cite{Muller, Pedregal, Puchala2} and references cited there for more details.

It is Laurence Chisholm Young who has first realized that there is need to enlarge the space of admissible functions to the space of more general objects. In \cite{Young}, the very first article devoted to the subject, these objects are called by the author 'generalized curves'. Today we call them 'Young measures'. In fact, "Young measure' is not a single measure, but a family of probability measures. More precisely, let there be given:
\begin{itemize}
\item
\mbox{$\varOmega$ -- a nonempty, bounded open subset of $\mathbb R^d$ with Lebesgue measure $M>0$;}
\item
a compact set $K\subset\mathbb R^l$;
\item
$(u_n)$ -- a sequence of measurable functions from $\varOmega$ to $K$, convergent to some function $u_0$ weakly$^{\ast}$ in an appropriate function space;
\item
$f$ -- an arbitrary continuous real valued function on $\mathbb{R}^d$.
\end{itemize}
The sequence $(f(u_n))$  has a subsequence  
weakly$^{\ast}$ convergent to some function $g$. However,  in general $g\neq f(u_0)$.  It can be proved that there exists a subsequence of $(f(u_n))$, not relabelled, and a family $(\nu_x)_{x\in\varOmega}$ of probability measures with supports $\textnormal{supp}\nu_x\subseteq K$, such that 
$\forall\,f\in C(\mathbb{R}^d)\;\forall w\in L^1(\varOmega)$ there holds
\[
\lim\limits_{n\to\infty}\int\limits_{\varOmega}f(u_n(x))w(x)dx=
\int\limits_{\varOmega}\int\limits_Kf(s)\nu_x(ds)w(x)dx:=
\int\limits_{\varOmega}\overline{f}(x)w(x)dx.
\]
This family of probability measures is today called a \emph{Young measure associated with the sequence} $(u_n)$. 

It often happens, both in theory and applications, that the Young measure is a 'one element family', that is it does not depend on $x\in\varOmega$. We call such Young measure a \emph{homogeneous} one. Homogeneous Young measures are the first and often the only examples of the Young measures associated with specific sequences of functions. They are also the main object of interest in this article.

Calculating an explicit form of particular Young measure is an important task. In elasticity theory, for example, Young measures can be regarded as means of the limits of minimizing sequences of the energy functional. These means summarize the spatial oscillatory properties of minimizing sequences, thus conserving some of that information; information that is entirely lost when quasiconvexification of the original functional is used as the minimization method. To quote from \cite{Muller}, page 121: \emph{The Young measure describes the local phase proportions in an infinitesimally fine mixture (modelled mathematically by a sequence that develops finer and finer oscillations)}, and a page earlier: \emph{The Young measure captures the essential feature of minimizing sequences}. It is worth noting, that the Young measures calculated in \cite{Muller}, in the context of elasticity theory or micromagnetism, are homogeneous ones. This indicates their role. This is also a case in many other books or articles devoted to the subject. Unfortunately, there is still no any general method of calculating their explicit form (homogeneous or not) when associated with specific, in particular minimizing, sequence. Some of the existing methods relay on periodic extensions of the elements of the sequence and application the generalized version of the Riemann-Lebesgue lemma or calculating weak$^\ast$ limits of function sequences.

In the article \cite{Puchala1} there has been proposed an elementary method of calculation an explicit form of Young measures. It is based on the notion of the 'quasi-Young measure' introduced there and the approach to Young measures as in \cite{Roubicek}. This approach enables us to look at them as at objects associated with \emph{any} measurable function defined on $\varOmega$ with values in $K$. Using only the change of variable theorem the explicit form of the quasi-Young measures for functions that are piecewise constant or piecewise invertible with differentiable inverses has been calculated. The result extends for sequences of oscillating functions of that shape. Finally, it has been proved that calculated quasi-Young measures are \emph{the} Young measures associated with considered (sequences of) functions. Significantly, all of them are homogeneous.

This article can be viewed as the continuation of \cite{Puchala1}. First, we recognize quasi-Young measures as homogeneous Young measures which are, in turn, constant mappings on $\varOmega$ with values in the space of probability measures on $K$. Then we provide the characterization of homogeneous Young measures associated with measurable functions form $\varOmega$ to $K$ as images of the normalized Lebesgue measure on $\varOmega$ with respect to the underlying functions. The proofs of these facts are really very simple, but it seems that such characterization of homogeneous Young measures has not been formulated yet. Then we show that the weak convergence of the densities of the homogeneous Young measures is equivalent to the weak convergence of these measures themselves. Here measures are understood as vectors in the Banach space of scalar valued measures with the total variation norm. The proof of this fact is also simple but relies on the nontrivial characterization of the weak sequential $L^1$ convergence of the sequence of functions.  In a special one-dimensional case the Young measure associated with function having differentiable inverse is a Lebesgue-Stieltjes measure and we can state the result for mappings having inverses of the $C^\infty$ class.

\section[]{Preliminaries and notation} 
We gather now some necessary information about Young measures. As it has been mentioned above, our approach is as in \cite{Roubicek}, where the material is presented in great detail. It is sketched in \cite{Puchala1}. See also the references cited in aforementioned articles.

Let $\varOmega$ be an open subset of $\mathbb{R}^d$ with Lebesgue measure $M>0$, $d\mu(x):=\tfrac 1 Mdx$, where $dx$ is the $d$-dimensional Lebesgue measure on $\varOmega$ and let $K\subset\mathbb{R}^l$ be compact. We will denote:
\begin{itemize}
\item
$rca(K)$ -- the space of regular, countably additive scalar measures on $K$, equipped with the norm $\|\rho\|_{rca(K)}:=\vert \rho\vert(\varOmega)$, where $\vert\cdot\vert$ stands in this case for the total variation of the measure $\rho$. With this norm $rca(K)$ is a Banach space, so we can consider a weak convergence of its elements;
\item
$rca^1(K)$ -- the subset of $rca(K)$ with elements being probability measures on $K$;
\item
$L_{w^\ast}^{\infty}(\varOmega ,\textnormal{rca}(K))$ -- the set of the weakly$^\ast$ measurable mappings
\[
\nu\colon\varOmega\ni x\to\nu(x)\in rca(K).
\]
(It follows that homogeneous Young measures are those from the above mappings, that are constant a.e. in $\varOmega$)
\end{itemize}
We equip this set with the norm
\[
\Vert\nu\Vert_{L_{w^\ast}^{\infty}(\varOmega ,\textnormal{rca}(K))}:=
\textnormal{ess}\sup\bigl\{\Vert\nu (x)\Vert_{\textnormal{rca}(K)}:
x\in\varOmega\bigr\}.
\]

Usually by a Young measure we understand a family 
$(\nu_x)_{x\in\varOmega}$ of regular proba\-bility Borel measures on $K$, indexed by elements $x$ of $\varOmega$ . However, 
we should remember that the Young measure is in fact a weakly$^\ast$ measurable mapping defined on an open set $\varOmega\subset\mathbb{R}^d$ having positive Lebesgue measure with values in the set $\textnormal{rca}^1(K)$.
The set of the Young measures on the compact set $K\subset\mathbb{R}^l$ will be denoted by $\mathcal{Y}(\varOmega ,K)$:
\[
\mathcal{Y}(\varOmega ,K):=\bigl\{\nu=(\nu (x))\in
L^{\infty}_{w^\ast}(\varOmega,\textnormal{rca}(K)):\nu_x\in
\textnormal{rca}^1(K)\;\textnormal{for a.a }x\in\varOmega\bigr\}.
\]
\begin{remark}
In \cite{Roubicek} Young measures are defined as \emph{weakly} measurable mappings, but it seems to be an innacuracy, since $rca(K)$ is in fact a conjugate space. Compare footnote 80 on page 36 in \cite{Roubicek} with, for example, definition 2.1.1 (d) on page109 in \cite{Gasinski}. However, when considering weak convergence of the Young measures, we will look at the $rca(K)$ as at the normed space itself with total variation norm.
\end{remark}
\begin{definition}
We say that a family $(\nu_x)_{x\in\varOmega}$ is a quasi-Young measure associa\-ted with the measurable function 
$u\colon\mathbb{R}^d\supset\varOmega\to K\subset\mathbb{R}^l$ if for every integrable function $\beta\colon K\to\mathbb{R}$ there holds
\[
\int\limits_{K}\beta(k)d\nu_x(k)=
\int\limits_{\varOmega}\beta(u(x))d\mu(x).
\]
We will write $\nu^u$ to indicate that the (quasi-)Young measure $\nu$ is associated with the function $u$.
\end{definition} 
\begin{remark}
This definition is slightly more general than the one formulated in \cite{Puchala1}. However, when dealing with Young measures, we must restrict ourselves to continuous functions $\beta$ to be able to use Riesz representation theorems. See \cite{Roubicek} for details.
\end{remark}
\section[]{Homogeneous Young measures}
We first prove that quasi-Young measures are constant mappings.
\begin{proposition}
Let $\nu^u$ be the quasi-Young measure associated with the Borel function $u$. Then $\nu^u$, regarded as a mapping from $\varOmega$ with values in $\textnormal{rca}(K)$, is constant.
\end{proposition}
\begin{proof} Suppose that $\nu^u$ is not a constant function. Then there exist 
$x_1,\,x_2\in\varOmega$, $x_1\neq x_2$, such that $\nu^u_{x_1}\neq\nu^u_{x_2}$. Then for any integrable $\beta\colon K\to\mathbb{R}$ there holds
\[
\int\limits_{K}\beta(k)d\nu^u_{x_1}(k)=\int\limits_{\varOmega}\beta(u(x))d\mu(x)=
\int\limits_{K}\beta(k)d\nu^u_{x_2}(k),
\]
a contradiction.
\end{proof}
\begin{corollary}
Let $\beta\colon K\to\mathbb{R}$ be a continuous function. Then the quasi-Young measure 
$\nu^u$ associated with the function $u$ is a Young measure, that is it is an element 
of the set $\mathcal{Y}(\varOmega ,K)$.
\end{corollary}
\begin{proof} Take $\beta\equiv 1$. Then 
\[
\int\limits_{K}d\nu^u(k)=\int\limits_{K}\beta(k)d\nu^u(k)=
\int\limits_{\varOmega}\beta(u(x))d\mu(x)=\int\limits_{\varOmega}d\mu(x)=1,
\]
which means that $\nu^u\in\textnormal{rca}^1(K)$. As a constant function it is weakly measurable, so that $\nu^u\in L_{w^\ast}^{\infty}(\varOmega ,\textnormal{rca}(K))$.
\end{proof}
\begin{corollary}
Quasi-Young measures associated with measurable functions $u\colon\varOmega\to K$ are precisely the homogeneous Young measures associated with them.
\end{corollary}
Due to the fact that homogeneous Young measures are constant functions, we will write '$\nu$' instead of '$\nu=(\nu_x)_{x\in\varOmega}$'.

Now we will prove that homogeneous Young measures are images of Borel measures with respect to the underlying functions. Recall that 
$u\colon\mathbb{R}^d\supset\varOmega\ni x\to u(x)\in K\subset\mathbb{R}^l$ and that $\mu$ is normed to unity Lebesgue measure on $\varOmega$.
\begin{theorem}
A homogeneous Young measure $\nu$ is associated with Borel function $u$ if and only if $\nu$ is an image of $\mu$ under $u$.
\end{theorem}
\begin{proof} The theorem follows from the equalities 
\[
\int\limits_{K}\beta(k)d\nu(k)=\int\limits_{\varOmega}\beta(u(x))d\mu(x)=\int\limits_{K}\beta(k)d\mu u^{-1}.
\]
\end{proof}
\section[]{Weak convergence} 
\subsection[]{Certain facts on weak convergence of functions and measures}
Let $(X,\mathcal A,\rho)$ be a measure space and consider a sequence $(v_n)$ of scalar functions defined on $X$ and integrable with respect to the measure $\rho$ (that is, $\forall n\in\mathbb N\; v_n\in L^1_\rho(X)$) and a function $v\in L^1_\rho(X)$. Recall that $(v_n)$ converges weakly sequentially to $v$ if
\[
\forall g\in L^\infty_\rho(X)\quad \lim\limits_{n\to\infty}\int\limits_Xv_ngd\rho=\int\limits_Xvgd\rho.
\]
Next theorem characterizes weak sequential $L^1$ convergence of functions and weak convergence of measures. We refer the reader to \cite{Czaja}, chapter 6 or \cite{Diestel}, chapter VII.
\begin{theorem}\label{CzajaDiestel}
\begin{itemize}
\item[(a)]
let $X$ be a locally compact Hausdorff space and $(X,\mathcal A,\rho)$ -- a measure space with $\rho$ regular. A sequence 
$(v_n)\subset L^1_\rho(X)$ converges weakly to some $v\in L^1_\rho(X)$ if and only if\, $\forall A\in\mathcal A$ the limit
\[
\lim\limits_{n\to\infty}\int\limits_Av_nd\rho
\]
exists and is finite;
\item[(b)]
let $X$ be a locally compact Hausdorff space and denote by $\mathcal B(X)$ the $\sigma$-algebra of Borel subsets of $X$. A sequence $(\rho_n)$ of scalar measures on $\mathcal B(X)$ converges weakly to some scalar measure $\rho$ on $\mathcal B(X)$ if and only if\, $\forall A\in\mathcal B(X)$ the limit
\[
\lim\limits_{n\to\infty}\rho_n(A)
\]
exists and is finite.
\end{itemize}
\end{theorem}
\subsection[]{Weak convergence of homogeneous Young measures and their densities}
Consider the function $u\colon\varOmega\to K$ of the form
\[
u:=\sum\limits_{i=1}^nu_i\chi_{\varOmega_i},
\]
where:
\begin{itemize} 
\item 
the open sets $\varOmega_i$, $i=1,2,\dots,n$, are pairwise disjoint and 
$\bigcup_{i=1}^n\overline{\varOmega}_i=\overline{\varOmega}$; denote this partition of $\varOmega$ by $\{\varOmega\}$;
\item
the functions $u_i\colon\varOmega_i\to K$, $i=1,2,\dots,n$, have continuously differentiable inverses $u_i^{-1}$ with Jacobians
$J_{u_i^{-1}}$;
\item
for any $i=1,2,\dots,n$ $\overline{u_i(\varOmega_i)}=K$.
\end{itemize}
\begin{theorem} \label{JaOpti}
(\cite{Puchala1}) The Young measure $\nu^u$ associated with $u$ is a homogeneous Young measure that is abolutely continuous with respect to the Lebesgue measure $dy$ on $K$. Its density is of the form
\[
g(y)=\frac 1 M\sum\limits_{i=1}^n\vert J_{u_i^{-1}}(y)\vert.
\]
\end{theorem}
Consider now a family $\{\varOmega^l\}$, $l\in\mathbb N$, with elements $\varOmega_i^l$, of open partitions of $\varOmega$  and a sequence $(u^l)$ of functions of the form $u^l:=\sum_{i=1}^{n(l)}u_i^l\chi_{\varOmega_i^l}$ satisfying, for fixed $l$, the above respective conditions. 
According to Theorem \ref{JaOpti} the Young measure associated with $u^l$ is absolutely continuous with respect to the Lebesgue measure $dy$ on $K$ with density
\[
g^l(y)=\frac 1 M\sum\limits_{i=1}^{n(l)}\vert J_{(u_i^l)^{-1}}(y)\vert.
\]
\begin{theorem}\label{equiconv}
Let $(u^l)$ be the sequence of functions described above and denote by $\nu^l$ the Young measure with density $g^l$ associated with the function $u^l$. Then the sequence $(g^l)$ is weakly convergent in $L^1(K)$ to some function $h$ if and only if the sequence $(\nu^l)$ is weakly convergent to some measure $\eta$.
\end{theorem}
\begin{proof}. ($\Rightarrow$) Since $(g^l)$ is weakly convergent in $L^1(K)$, then for any measurable $A\subseteq K$ the limit
\[
\lim\limits_{l\to\infty}\int\limits_Ag^ldy
\]
exists and is finite. This in turn is equivalent to the fact that the sequence $(\nu^l)$ of Young measures is weakly convergent to some measure $\eta$.\\
$(\Leftarrow)$ We proceed as above, but start the reasoning from the weak convergence of the sequence $(\nu^l)$.
\end{proof}

Let $I$ be an open interval in $\mathbb R$ with Lebesgue measure $M>0$ and $u$ -- a strictly monotonic differentiable real valued function on $I$. We can assume that $u$ is strictly increasing. Then the set 
$K:=\overline{u(I)}$ is compact and for any $\beta\in C(K,\mathbb R)$ we have
\[
\int\limits_K\beta(y)\tfrac 1 M(u^{-1})'(y)dy=\int\limits_K\beta(y)d\tfrac 1 Mu^{-1}(y).
\]
This means that the (homogeneous) Young measure associated with $u$ is a Lebesgue-Stieltjes measure on $K$. This observation together with standard inductive argument lead to the following one-dimensional corollary to theorem \ref{equiconv}.
\begin{corollary}
Let $(u_n)$ be a sequence of real valued functions from a bounded, nondegenerate interval $I\subset\mathbb R$ such that: 
\begin{itemize}
\item[(i)]
for any $n\in\mathbb N$ $u_n$ is a $C^\infty$-diffeomorphism;
\item[(ii)]
$\overline{u_n^{(l)}(I)}:=K_l$ is compact, $l\in\mathbb N\cup\{0\}$;
\item[(iii)]
$u_n^{(l)}$ is strictly positive or negative, $n\in\mathbb N$.
\end{itemize}
Then for any fixed $l\in\mathbb N$ the weak  $L^1$ convergence of the sequence $\bigl((u_n^{-1})^{(l)}\bigr)$ is equivalent to the weak convergence of the sequence of the Young-Lebesgue-Stieltjes measures associated with the elements of the sequence $\bigl((u_n^{-1})^{(l)}\bigr)$.
\end{corollary}
\begin{example}(see \cite{Puchala3})
Let $I:=(a,b)$, $a<b$, $K:=[c,d]$, $c<d$ and let $(u_n)$ be a sequence of strictly monotonic functions from $I$ with values in $K$ such that:
\begin{itemize}
\item[(i)]
$\forall\,n\in\mathbb N\;\;\overline{u_n(I)}=K$;
\item[(ii)]
$\forall\,n\in\mathbb N$ $u_n$ has continuously differentiable inverse $(u_n^{-1})'$;
\item[(iii)]
the sequence $(u_n^{-1})'$ is nondecreasing.
\end{itemize}
By monotonicity of the sequence of derivatives for any Borel subset $A\subseteq K$ and any natural numbers $m\leq n$ we have
\[
\int\limits_A(u_m^{-1})'(y)dy\leq\int\limits_A(u_n^{-1})'(y)dy.
\]
Further, $\forall\,n\in\mathbb N$ the Young measure associated with function $(u_n^{-1})'$ is a homogeneous one. Since it is probability measure on $K$, the sequence $\Bigl(\int\limits_A(u_n^{-1})'(y)dy\Bigr)$ is monotonic and bounded. Thus for any Borel subset $A\subseteq K$ the limit
\[
\lim\limits_{n\to\infty}\int\limits_A(u_n^{-1})'(y)dy=\lim\limits_{n\to\infty}\int\limits_Adu_n^{-1}(y)
\]
exists and is finite, which by theorem \ref{CzajaDiestel} yields the weak $L^1$ convergence of the sequence
$\bigl((u_n^{-1})'\bigr)$. This, by theorem \ref{equiconv}, is equivalent to the weak convergence of the sequence of Young measures, whose elements are associated with respective elements of $\bigl((u_n^{-1})'\bigr)$.
\end{example} 

\textbf{Acknowledgements.}
Author would like to thank Professor Agnieszka Ka{\l}amajska for discussion that took place during Dynamical Systems and Applications IV conference in June 2016, {\L}\'od\'z, Poland.


\end{document}